\documentclass[12pt]{amsart}

\usepackage{amsmath}
\usepackage{amsthm}
\usepackage{amssymb}
\usepackage{amsfonts}
\usepackage{amsrefs}
\usepackage{color}
\usepackage{tikz}

\tikzstyle{shaded}=[fill=red!10!blue!20!gray!30!white]
\tikzstyle{shaded line}=[double=red!10!blue!20!gray!30!white, double distance=1.5mm, draw=black]
\tikzstyle{unshaded}=[fill=white]
\tikzstyle{unshaded line}=[double=white, double distance=1.5mm, draw=black]
\tikzstyle{Tbox}=[circle, draw, thick, fill=white, opaque,]
\tikzstyle{empty box}=[circle, draw, thick, fill=white, opaque, inner sep=2mm]
\tikzstyle{background rectangle}= [fill=red!10!blue!20!gray!40!white,rounded corners=2mm] 
\tikzstyle{on}=[very thick, red!50!blue!50!black]
\tikzstyle{off}=[gray]

\tikzstyle{traces}=[scale=.2, inner sep=1mm]
\tikzstyle{quadratic}=[scale=.4, inner sep=1mm, baseline]
\tikzstyle{annular}=[scale=.7, inner sep=1mm, baseline]
\tikzstyle{make triple edge size}= [scale=.4, inner sep=1mm,baseline] 
\tikzstyle{icosahedron network}=[scale=.3, inner sep=1mm, baseline]
\tikzstyle{ATLsix}=[scale=.25, baseline]
\tikzstyle{TL12}=[scale=.15,baseline]
\tikzstyle{PAdefn}=[scale=.7,baseline]
\tikzstyle{TLEG}=[scale=.5,baseline]

\makeatletter

\newtheorem{lemma}{Lemma}[section]
\newtheorem{definition}[lemma]{Definition}
\newtheorem{theorem}[lemma]{Theorem}

\newtheorem{remark}[lemma]{Remark}
\newtheorem{corollary}[lemma]{Corollary}

\newtheorem{question}[lemma]{Question}

\newenvironment{claim}[1]{\par\noindent\underline{Claim:}\space#1}{}
\newenvironment{claimproof}[1]{\par\noindent\underline{Proof:}\space#1}{\hfill $\blacksquare$}

\makeatother

\title[Dual Ore's theorem on distributive intervals]{Dual Ore's theorem on distributive intervals of finite groups}
\author[Sebastien Palcoux]{Sebastien Palcoux}
\address{Institute of Mathematical Sciences, Chennai, India}
\email{sebastienpalcoux@gmail.com}
\subjclass[2010]{20D30, 05E15, 20C15, 06C15}
\keywords{finite group; complex representation; subgroup lattice; distributive; Boolean}

\begin{document}

\begin{abstract}
This paper gives a self-contained group-theoretic proof of a dual version of a theorem of Ore on distributive intervals of finite groups. We deduce a bridge between combinatorics and representations in finite group theory.
\end{abstract}

\maketitle
\section{Introduction} 
\O ystein Ore proved in 1938 that a finite group is cyclic if and only if its subgroup lattice is distributive, and he extended one side as follows, where $[H,G]$ will be an interval in the subgroup lattice of the group $G$ (idem throughout the paper).

\begin{theorem}[\cite{or}] \label{oreintro}
Let $[H,G]$ be a distributive interval of finite groups. Then there is $g \in G$ such that $\langle  Hg \rangle = G$.
\end{theorem}

\noindent This paper first recalls our short proof of Theorem \ref{oreintro} and then gives a self-contained group-theoretic proof of the following dual version, where $G_{(V^H)}$ will be the pointwise stabilizer subgroup of $G$ for the fixed-point subspace $V^H$ (see Definition \ref{fixstab}).

\begin{theorem} \label{dualoreintro}  Let $[H,G]$ be a distributive interval of finite groups. Then there exists an irreducible complex representation $V$ of $G$ such that $G_{(V^H)}=H$.
\end{theorem}
\noindent We deduce a bridge between combinatorics and representations:

\begin{corollary} \label{corointro}
The minimal number of irreducible components for a faithful complex representation of a finite group $G$ is at most the minimal length $\ell$ for an ordered chain of subgroups $$\{e\}=H_0 < H_1 < \dots < H_{\ell} = G$$ such that $[H_i,H_{i+1}]$ is distributive (or better, bottom Boolean). 
\end{corollary}
\noindent It is a non-trivial upper bound involving the subgroup lattice only. These results were first proved by the author as applications to finite group theory of results on planar algebras \cite[Corollaries 6.10, 6.11]{p}. For the convenience of the reader and for being self-contained, this paper reproduces some preliminaries of \cite{bp} and \cite{p}.

\section{Ore's theorem on distributive intervals}  
\subsection{Basics in lattice theory} \label{baslat}
We refer to \cite{sta} for the notions of \emph{finite lattice} $L$, \emph{meet} $\wedge$, \emph{join} $\vee$, \emph{subgroup lattice} $\mathcal{L}(G)$, \emph{sublattice} $L' \subseteq L$, \emph{interval} $ [a,b] \subseteq L$, \emph{minimum} $\hat{0}$, \emph{maximum} $\hat{1} $, \emph{atom}, \emph{coatom}, \emph{distributive lattice}, \emph{Boolean lattice} $\mathcal{B}_n$ (of rank $n$) and \emph{complement} $b^{\complement}$ (with $b \in \mathcal{B}_n$). The \emph{top interval} of a finite lattice $ L $ is the interval $ [t,\hat{1}] $, with $ t $ the meet of all the coatoms. The \emph{bottom interval} of a finite lattice $ L $ is the interval $ [\hat{0},b] $, with $ b $ the join of all the atoms. A lattice with a Boolean top interval will be called \emph{top Boolean}; idem for \emph{bottom Boolean}.

\begin{lemma} \label{topBn}
A finite distributive lattice is top and bottom Boolean.
\end{lemma} 
\begin{proof}
See \cite[items a-i p254-255]{sta} which uses Birkhoff's representation theorem (a finite lattice is distributive if and only if it embeds into some $ \mathcal{B}_n $).
\end{proof} 
\subsection{The proof} \hspace*{1cm} \\
\noindent {\O}ystein Ore proved the following result in \cite[Theorem 4, p267]{or}.
\begin{theorem} \label{ore1}
 
A finite group $ G $ is cyclic if and only if its subgroup lattice $ \mathcal{L}(G) $ is distributive. 
\end{theorem} 
\noindent Theorem \ref{oreintro} is an extension by Ore of one side of Theorem \ref{ore1} to any distributive interval of finite groups \cite[Theorem 7, p269]{or}.
 \begin{definition} \label{Hcy}
An interval of finite groups $ [H,G] $ is said to be $ H $-cyclic if there is $ g \in G $ such that $ \langle  H,g \rangle = G $. Note that $ \langle  H,g \rangle = \langle  Hg \rangle $. 
\end{definition}
\noindent We will give our short alternative proof of Theorem \ref{oreintro} by extending it to any top Boolean interval (see Lemma \ref{topBn}) as follows:
\begin{theorem} \label{ore2}
A top Boolean interval $ [H,G] $ is $H$-cyclic.
\end{theorem} 
\begin{proof}
The proof follows from the claims below.
\\
\begin{claim} \label{max} Let $ M $ be a maximal subgroup of $ G $. Then $ [M,G] $ is $ M $-cyclic.
\end{claim}
\begin{claimproof}
For $ g \in G $ with $ g \not \in M $, we have $ \langle  M,g \rangle = G $ by maximality.
\end{claimproof}
\\
\begin{claim} \label{preore2}
A Boolean interval $ [H,G] $ is $ H $-cyclic.
\end{claim} 
\begin{claimproof}
Let $ M $ be a coatom in $ [H,G] $, and $ M^{\complement} $ be its complement. By the previous claim and induction on the rank of the Boolean lattice, we can assume $ [H,M] $ and $ [H,M^{\complement}] $ both to be $ H $-cyclic, i.e. there are $ a, b \in G $ such that $ \langle  H,a \rangle = M $ and $ \langle  H,b \rangle = M^{\complement} $. For $ g=a b $, $ a=g b^{-1} $ and $ b=a^{-1}g $, so $ \langle  H,a,g \rangle = \langle  H,g,b \rangle = \langle  H,a,b \rangle = M \vee M^{\complement} = G $. 
Now, $ \langle  H,g \rangle = \langle  H,g \rangle \vee H = \langle  H,g \rangle \vee (M \wedge M^{\complement}) $ but by distributivity $ \langle  H,g \rangle \vee (M \wedge M^{\complement}) = (\langle  H,g \rangle \vee M \rangle) \wedge (\langle  H,g \rangle \vee M^{\complement} \rangle) $. So $ \langle  H,g \rangle = \langle  H,a,g \rangle \wedge \langle  H,g,b \rangle = G $. The result follows.
\end{claimproof} 
\\
\begin{claim} \label{topred} 
 $ [H , G] $ is $ H $-cyclic if its top interval $ [K,G] $ is $K$-cyclic.
\end{claim}
\begin{claimproof}
Consider $ g \in G $ with $ \langle  K,g \rangle = G $. For any coatom $ M \in [H,G] $, we have $ K \subseteq M $ by definition, and so $ g \not \in M $, then a fortiori $ \langle  H,g \rangle \not \subseteq M $. It follows that $ \langle  H,g \rangle=G $.
\end{claimproof} 
\end{proof}
\noindent The converse is false because $ \langle  S_2, (1234) \rangle = S_4 $ whereas $ [S_2,S_4] $ is not top Boolean.

\section{Dual Ore's theorem on distributive intervals}
\subsection{Basics in Galois connections}
\begin{definition} \label{fixstab} Let $W$ be a representation of a group $G$, $K$ a subgroup of $G$, and $X$ a subspace of $W$. We define the \emph{fixed-point subspace} $$W^{K}:=\{w \in W \ \vert \  kw=w \ , \forall k \in K  \},$$ and the \emph{pointwise stabilizer subgroup} $$G_{(X)}:=\{ g \in G \  \vert \ gx=x \ , \forall x \in X \}.$$  \end{definition} 

\begin{lemma} \label{tech} Let $G$ be a finite group, $H,K$ two subgroups, $V$ a representation of $G$ and $X,Y$ two subspaces of $V$. Then
\begin{itemize}
\item[(1)] $H \subseteq K \Rightarrow V^{K} \subseteq V^{H}$,
\item[(2)] $X \subseteq Y \Rightarrow G_{(Y)} \subseteq G_{(X)}$,
\item[(3)] $V^{H \vee K} = V^H \cap V^K $,
\item[(4)] $H \subseteq G_{(V^H)}$,
\item[(5)] $ V^{G_{(V^H)}} = V^H$,
\item[(6)] $[H \subseteq K$ and $V^{K} \subsetneq V^{H}]$ $\Rightarrow K \not \subseteq G_{(V^H)}$.
\end{itemize}
\end{lemma}

\begin{proof} (1) and (2) are immediate. 

\noindent (3) First $H,K \subseteq H \vee K$, so $V^{H \vee K} $ is included in $ V^H$ and  $V^{K}$, so in $V^H \cap V^{K}$. Now take $v \in V^H \cap V^{K}$, then $\forall h \in H$ and $\forall k \in K$, $hv = kv = v$, but any element $g \in H \vee K$ is of the form $h_1k_1h_2k_2 \cdots h_rk_r $ with $h_i \in H$ and $k_i \in K$, it follows that $gv=v$ and so $V^H \cap V^{K} \subseteq V^{H \vee K}$. (4) Take $h \in H$ and $v \in V^H$. Then by definition $hv = v$, so $H \subseteq  G_{(V^H)}$.  (5) From (1) and (4) we deduce that $V^{G_{(V^H)}} \subseteq V^H$. Now take $v \in V^H$ and $g \in G_{(V^H)}$, by definition $gv = v$, so $V^H \subseteq  V^{G_{(V^H)}}$ also. \\
(6) Suppose that $K \subseteq G_{(V^H)}$, then $V^K \supseteq V^{G_{(V^H)}} = V^H$ by (1) and (5). Hence $V^K = V^H$ by (1), contradiction with $V^K \subsetneq V^H$.
\end{proof}

\subsection{Induced representation}
\begin{definition} Let $G$ be a finite group and $H$ a subgroup. Consider the set $G / H = \{g_1H, \dots , g_sH \}$, with $g_1=e$. Let $V$ be a complex representation of $H$. The \emph{induced representation} $\mathrm{Ind}_H^G(V)$ is a space $\bigoplus_i g_iV$ (we identify $eV$ with $V$) on which $G$ acts as follows:
$$ g \cdot (\sum_i g_iv_i) = \sum_i g_{\tau(i,g)}(h_{i,g} \cdot v_i)  $$ with $gg_i = g_{\tau(i,g)}h_{i,g}$, $\tau(i,g) \in \{1, \dots , s \}$ and $h_{i,g} \in H$.
\end{definition}

Let $\langle  \cdot,\cdot \rangle_G$ be the usual normalized inner product of finite dimensional complex representations (up to equivalence) of a finite group $G$. 
\begin{lemma}[Frobenius reciprocity, \cite{isa} p62]
Let $G$ be a finite group and $H$ a subgroup. Let $V$ (resp. $W$) be a finite dimensional complex  representation of $G$ (resp. of $H$). Let $\mathrm{Ind}(W)$ be the induction to $G$ and $\mathrm{Res}(V)$ the restriction to $H$, then $\langle V,\mathrm{Ind}(W)\rangle_G = \langle \mathrm{Res}(V),W\rangle_H$. 
\end{lemma}

\begin{lemma} \label{indexrep}
Let $[H,G]$ be an interval of finite groups. Let $V_1, \dots , V_r$ be  the irreducible complex representations of $G$ (up to equivalence). Then $$|G:H| = \sum_{i=1}^r \dim(V_i)\dim(V_i^H).$$   
\end{lemma} 
\begin{proof} The following proof is due to Tobias Kildetoft.
Let $1_H^G$ be the trivial representation of $H$ induced to $G$. On one hand, it has dimension $|G:H|$, and on the other hand, this dimension is also $$\sum_i \dim(V_i)\langle V_i,1_H^G \rangle_G = \sum_i\dim(V_i)\langle V_i,1_H \rangle_H = \sum_i\dim(V_i)\dim(V_i^H).$$ The first equality follows from Frobenius reciprocity.
\end{proof} 

\begin{lemma} \label{nonzero}
Let $[H,G]$ be an interval of finite groups with $H \neq G$. Then there is a non-trivial irreducible complex representation $V$ of $G$ such that $V^H \neq 0$.   
\end{lemma} 
\begin{proof}
By Lemma \ref{indexrep}, $ \sum_i \dim(V_i)\dim(V_i^H) = |G:H| \ge 2$ because $H \neq G$. The result follows.
\end{proof}

\begin{lemma} \label{indres}
Let $G$ be a finite group, $K$ a subgroup and $U$ an irreducible complex representation of $K$. For any irreducible component $V$ of the induction $W=\mathrm{Ind}_K^G(U)$, there exists $\tilde{U} \subseteq V$ equivalent to $U$ as a representation of $K$.
\end{lemma} 
\begin{proof}
Direct by Frobenius reciprocity because $\langle V,W \rangle_G = \langle V,U \rangle_K$.
\end{proof}

\begin{lemma} \label{cliff}
Let $[H,G]$ be an interval of finite groups, $K \in [H,G]$ and $U$ an irreducible complex representation of $K$ such that $U^H \neq 0$. Let $W$ be the induction $\mathrm{Ind}_K^G(U)$. Let $V$ be an irreducible component of $W$. Then
\begin{itemize}
\item[(1)] $G_{(W^H)} \subseteq  K_{(U^H)}$,
\item[(2)] $V^H \neq 0$ and $K_{(V^H)} \subseteq K_{(U^H)}$,
\item[(3)] if $K_{(U^H)}=H$ then $G_{(W^H)} = K_{(V^H)} = H$.
\end{itemize} 
\end{lemma} 
\begin{proof}
(1) $U^H \subseteq W^H$, so by Lemma \ref{tech}(2), $G_{(W^H)} \subseteq G_{(U^H)}$. Now, $U^H \neq 0$, so by definition of the induction, $g \cdot U^H  \subseteq U$ if and only if $gK = K$, if and only if $g \in K$. Thus $G_{(U^H)} \subseteq K_{(U^H)}$, and so $G_{(W^H)} \subseteq K_{(U^H)}$. 

\noindent
(2) Take $\tilde{U}$ as for Lemma \ref{indres}. Then $0 \neq \tilde{U}^H \subseteq V^H$ and $K_{(V^H)} \subseteq K_{(\tilde{U}^H)} = K_{(U^H)}$.

\noindent
(3) By (1) and Lemma \ref{tech}(4), $$H \subseteq G_{(W^H)} \subseteq  K_{(U^H)} = H,$$ so $G_{(W^H)} = H$. Idem, by (2) and Lemma \ref{tech}(4), $K_{(V^H)} = H$.
\end{proof}

\subsection{The proof}
\begin{definition} \label{linprimdef}
The group $G$ is called \emph{linearly primitive} if it admits an irreducible complex representation $V$ which is faithful, i.e. $G_{(V)} = \{ e \}$.
\end{definition}

\begin{definition} \label{linprim}
The interval $[H,G]$ is called \emph{linearly primitive} if there is an irreducible complex representation $V$ of $G$ such that $G_{(V^H)} = H$.
\end{definition}

\begin{lemma} \label{maxlin}  A maximal interval $[H,G]$ is linearly primitive.
\end{lemma}
\begin{proof}
By Lemma \ref{nonzero}, there is a non-trivial irreducible complex representation $V$ of $G$ with $V^H \neq 0$. By Lemma \ref{tech}(4), $H \subseteq G_{(V^H)}$. If $G_{(V^H)} = G$ then $V$ must be trivial (by irreducibility), so by maximality $G_{(V^H)} = H$.
\end{proof}

\begin{lemma} \label{botlin} 
The interval $[H,G]$ is linearly primitive if its bottom interval $[H,K]$ is so.
\end{lemma}
\begin{proof}
Let $[H,K]$ be the bottom interval of $[H,G]$, i.e. $K=\bigvee_i K_i$ with $K_1, \dots, K_n$ the atoms of $[H,G]$. By assumption, there is an irreducible complex representation $U$ of $K$ such that $K_{(U^H)} = H$. Let $V$ be an irreducible component of $\mathrm{Ind}_K^G(U)$. By Lemma \ref{cliff}(3), $K_{(V^H)}=H$. Now if $\exists i$ such that $V^H = V^{K_i}$, then $$ K_i \subseteq  K_{(V^{K_i})}  = K_{(V^H)} = H,$$ contradiction with $H \subsetneq K_i$. So $\forall i$ $V^{K_i} \subsetneq V^H$. By Lemma \ref{tech}(6), we deduce that $K_i \not \subseteq G_{(V^H)}$ $\forall i$, so by minimality $G_{(V^H)} = H$.   
\end{proof}
 
\noindent A dual version of Theorem \ref{ore2} is the following:
\begin{theorem} \label{dualore} A bottom Boolean interval $[H,G]$ is linearly primitive.
\end{theorem}
\begin{proof}
By Lemma \ref{botlin}, we can reduced to Boolean intervals. We make an induction on the rank of the Boolean lattice. The rank
one case is handled in Lemma \ref{maxlin}. Assume that it is true at rank $<n$. We will write a proof at rank $n \ge 2$. Let $K$ be a coatom of $[H,G]$. Then $[H,K]$ is Boolean of rank $n-1$, so by assumption, it is linearly primitive, thus there is an irreducible complex representation $U$ of $K$ such that $K_{(U^H)} = H$. For any irreducible component $V$ of $W=\mathrm{Ind}_K^G(U)$, we have $$ K \wedge G_{(V^H)} = K_{(V^H)} = H$$ 
by Lemma \ref{cliff}(3). Thus, by the Boolean structure, $G_{(V^H)} \le K^{\complement}$ because $$K^{\complement} = K^{\complement} \vee H  = K^{\complement} \vee (K \wedge G_{(V^H)}) = (K^{\complement} \vee K) \wedge (K^{\complement} \vee G_{(V^H)}) = K^{\complement} \vee G_{(V^H)}.$$ But $K$ is a coatom of $[H,G]$, so $K^{\complement}$ is an atom and then $$G_{(V^H)} \in \{H,K^{\complement}\}.$$ 
\noindent Assume that every irreducible component $V$ of $W$ satisfies $G_{(V^H)} = K^{\complement}$. There are irreducible complex representations $V_1, \dots , V_r$ of $G$ such that $W = \bigoplus_i V_i$, then by Lemma \ref{cliff}(3)
$$H = G_{(W^H)} = \bigwedge_i G_{(V_i^H)}  =  K^{\complement},$$ 
thus $H = K^{\complement}$, contradiction. So there is an irreducible component $V$ of $W$ such that $G_{(V^H)} = H $, and the result follows. 
\end{proof} 
\noindent Theorem \ref{dualoreintro} follows directly from Theorem \ref{dualore} and Lemma \ref{topBn}.
\section{A bridge between combinatorics and representations} 
\noindent We restate Corollary \ref{corointro} by using the notation $H^i$ instead of $H_i$ because it will be more convenient for the proof.
\begin{corollary} \label{cor}
The minimal number of irreducible components for a faithful complex representation of a finite group $G$ is at most the minimal length $\ell$ for an ordered chain of subgroups $$\{e\}=H^0 < H^1 < \dots < H^{\ell} = G$$ such that $[H^i,H^{i+1}]$ is distributive (or better, bottom Boolean). 
\end{corollary}
\begin{proof}
By Theorem \ref{dualore} and Lemma \ref{cliff}(3), there are irreducible complex representations $V_1, \dots , V_{\ell}$ of $G$ such that $H^{i}_{(V_i^{H^{i-1}})} = H^{i-1}$. Take $W=\bigoplus_{i=1}^{\ell} V_i$. Then  $$\ker(\pi_W) = \bigwedge_{i=1}^{\ell}\ker(\pi_{V_i}) = \bigwedge_{i=1}^{\ell} G_{(V_i)} \le  \bigwedge_{i=1}^{\ell} G_{(V_i^{H^{i-1}})} = (\bigwedge_{i=1}^{\ell-1} G_{(V_i^{H^{i-1}})}) \wedge H^{\ell - 1} $$ $$ = \bigwedge_{i=1}^{\ell-1} H_{(V_i^{H^{i-1}})}^{\ell - 1} = \cdots = \bigwedge_{i=1}^{\ell-s} H^{\ell - s}_{(V_i^{H^{i-1}})} = \cdots = H^0 = \{e\}.$$ 
So $W$ is faithful with $\ell$ irreducible components. The result follows.
\end{proof}
\noindent Note that this upper bound involves the subgroup lattice only.
\begin{remark} \label{mod}  The modular maximal-cyclic group $M_4(2)$ and the abelian group $C_8 \times C_2$ have the same subgroup lattice, but the first is linearly primitive whereas the second is not. So the minimal number in Corollary \ref{cor} cannot be determined by the subgroup lattice only.
\end{remark} 
\begin{lemma} \label{linprimgrp} For $H$ core-free, $G$ is linearly primitive if $[H,G]$ is so. 
\end{lemma}
\begin{proof} Let $V$ be an irreducible complex representation of $G$ such that $G_{(V^H)}=H$. Now, $V^{H} \subseteq V$ so $G_{(V)} \subseteq G_{(V^H)}$, but $\ker(\pi_V) =  G_{(V)}$, it follows that $\ker(\pi_V) \subseteq H$; but $H$ is a core-free subgroup of $G$, and $\ker(\pi_V)$ a normal subgroup of $G$, so $\ker(\pi_V)= \{ e \}$. \end{proof}
\noindent By Lemma \ref{linprimgrp}, we can improve the bound of Corollary \ref{cor} by taking for $H^0$ any core-free subgroup of $H^1$ (instead of just $\{e\}$), and we can wonder whether it is the right answer in general, in particular:
\begin{question}
Is a finite group $G$ linearly primitive if and only if there is a core-free subgroup $H$ with $[H,G]$ bottom Boolean?
\end{question}
\noindent It is true for any finite simple group $S$, because any proper subgroup $M$ is core-free, and by choosing it maximal, $[M,S]$ is Boolean of rank one. Moreover, we have checked by GAP \cite{gap} that it is also true for any finite group $G$ of order less than $512$.
\begin{remark} A normal subgroup $N \trianglelefteq G$ is a modular element in $\mathcal{L}(G)$, see \cite[p43]{sch}. If $H$ is a subgroup of $G$ such that $\forall K \in (1,H]$, $K$ is not modular in $\mathcal{L}(G)$, then $H$ is core-free. It follows that we can also improve the bound of Corollary \ref{cor} completely combinatorially. Nevertheless, by Remark \ref{mod}, it cannot be the right answer in general. 
\end{remark} 

\section{Acknowledgments} 
\noindent This work is supported by the Institute of Mathematical Sciences, Chennai. The author thanks an anonymous referee who suggested to write this group-theoretic proof of dual Ore's theorem.

\end{document}